\documentclass[11pt]{amsart}

\usepackage{amsfonts, amstext, amsmath, amsthm, amscd, amssymb}

\usepackage{graphicx, color}

\usepackage{microtype}

\usepackage[hidelinks,pagebackref,pdftex]{hyperref}

\usepackage[margin=3cm]{geometry}

\newtheorem{theorem}{Theorem}[section]
\newtheorem{proposition}[theorem]{Proposition}
\newtheorem{lemma}[theorem]{Lemma}
\newtheorem{corollary}[theorem]{Corollary}

\newtheorem*{namedtheorem}{\theoremname}
\newcommand{\theoremname}{testing}

\theoremstyle{definition}
\newtheorem{definition}[theorem]{Definition}

\newtheorem{remark}[theorem]{Remark}

\title[]{Lorenz Links that are not Horseshoe and Rossler Links}
\author{Thiago de Paiva}
\address[]{Beijing International Center for Mathematical Research, Peking University, Beijing 100871, China P.R.}
\email[]{thhiagodepaiva@gmail.com}
\author{Yi Liu}
\address{Beijing International Center for Mathematical Research, Peking University, Beijing 100871, China P.R.}
\email{liuyi@bicmr.pku.edu.cn}

\begin{document}

\begin{abstract}
We compare the periodic orbit types of the Lorenz, horseshoe, and R\"ossler
systems through their associated templates. Lorenz links are carried by the
Lorenz template, while horseshoe links are carried by the horseshoe template.
For the standard template model considered here, the R\"ossler template is
identified, up to inversion symmetry, with the template of the horseshoe
mechanism. Thus, in this paper, R\"ossler links and horseshoe links are treated
as belonging to the same template class.

We show that the relationship between Lorenz links and horseshoe/R\"ossler
links has two complementary sides. First, we prove that the overlap between
the two families is nontrivial by constructing infinite families of
links which can be embedded in both templates. This verifies, for these families, a conjecture stated by Kofman 
that horseshoe links should also be Lorenz links. 
On the other hand, we prove that the two
families are far from being the same. Holmes and Williams showed that many
Lorenz torus knots cannot be embedded in the horseshoe template: if the torus
knot \(T(p,q)\), with \(p<q\), is a horseshoe knot, then \(3p\leq 2q\).
We show that this phenomenon is much broader. We extend the Holmes--Williams
obstruction from torus knots to torus links, and we construct infinitely many
hyperbolic Lorenz knots and links, as well as infinitely many satellite Lorenz
knots and links, which cannot be embedded in the horseshoe template.
Consequently, these examples are Lorenz links which are not horseshoe links
and hence not R\"ossler links.
\end{abstract}

\maketitle

\section{Introduction}

One of the classical ways of studying chaotic dynamical systems is through
the knot and link types of their periodic orbits. In several important
examples, these periodic orbits are carried by templates, that is, embedded
branched surfaces endowed with a semiflow. Thus a dynamical problem about
periodic orbits can be translated into a topological problem about links
embedded in a template.

The three motivating examples in this paper are the Lorenz system, the
suspension of Smale's horseshoe map, and the R\"ossler system. We do not study
these systems for arbitrary parameter values. Rather, we consider the
classical chaotic regimes in which their periodic orbits are modeled by
templates.

The Lorenz system is the flow in \(\mathbb R^3\) defined by
\[
\dot{x}=\sigma(y-x),\qquad
\dot{y}=x(\rho-z)-y,\qquad
\dot{z}=xy-\beta z.
\]
For the classical chaotic parameters, such as
\[
\sigma=10,\qquad \rho=28,\qquad \beta=\frac{8}{3},
\]
the periodic orbits on the Lorenz attractor are modeled by the Lorenz
template; see Figure~\ref{LT}. Links carried by this template are called
Lorenz links.

The study of Lorenz links has its roots in the foundational work of Birman
and Williams \cite{periodicorbits,birman1983knotted}, who showed that the
periodic orbits of the Lorenz system can be studied using the Lorenz template.
Birman and Kofman later gave a braid-theoretic description of Lorenz links in
terms of T-links \cite{newtwis}. Further developments on Lorenz links can
be found in
\cite{dePaivaPurcell2024,de2022torus,dePaiva2025,de2021satellites,unexpected}.

The Smale horseshoe is a diffeomorphism of the disk with a chaotic invariant
set. By taking a suspension, one obtains a flow whose periodic orbits are
modeled by the horseshoe template; see Figure~\ref{HT}. Links carried by
this template are called horseshoe links. The braid types of periodic orbits
in the horseshoe template have been studied extensively. In particular,
de Carvalho and Hall studied forcing relations, conjugacies, and decoration
invariants for horseshoe braids
\cite{deCarvalhoHall2002Forcing,deCarvalhoHall2003Conjugacies,
deCarvalhoHall2010Decoration}.

The R\"ossler system is the flow in \(\mathbb R^3\) defined by
\[
\dot{x}=-y-z,\qquad
\dot{y}=x+ay,\qquad
\dot{z}=b+z(x-c).
\]
For classical chaotic parameter values, the topology of its periodic orbits
can be described by a template, usually called the R\"ossler template
\cite{LetellierDutertreMaheu1995}. For recent work on knots and chaos in the
R\"ossler system, see \cite{Igra2025Rossler}. For the standard template model
considered in this paper, the R\"ossler template is identified, up to inversion
symmetry, with the template of the horseshoe mechanism
\cite{LetellierGouesbet1996,RosalieLetellier2013}. Therefore, throughout this
paper, R\"ossler links and horseshoe links are treated as belonging to the same
template class.

From this perspective, comparing Lorenz links with horseshoe and R\"ossler
links is a way of comparing the periodic orbit types of the corresponding
systems. If two flows are orbit equivalent, then there is a homeomorphism
between their ambient manifolds which sends orbits of one flow to orbits of
the other. Hence periodic orbits are sent to periodic orbits, and complements
of periodic orbit links are sent to complements of periodic orbit links.
Therefore, differences between the link types realized by two templates give
topological obstructions to orbit equivalence. This provides one of the
motivations for studying which Lorenz links can, or cannot, be realized in
the horseshoe template.

More generally, whenever two natural families of knots and links arise from
dynamical systems, a basic problem is to compare their link spectra. One would
like to understand which links belong to both families, which links belong to
one family but not the other. Thus the comparison between Lorenz links and horseshoe/R\"ossler links
fits into a broader classification problem for links carried by templates.

We use the following terminology. A link \(L\subset S^3\) is called a
\emph{Lorenz link}, \emph{horseshoe link}, or \emph{R\"ossler link} if it is
carried by the Lorenz, horseshoe, or R\"ossler template, respectively; that is,
if it can be realized as the link of a finite collection of periodic orbits on
the corresponding template.

Throughout the paper, links are considered up to ambient isotopy in \(S^3\).
Thus two links are said to be equivalent if there is an ambient isotopy of
\(S^3\) taking one link to the other.

For the standard template model considered in this paper, the R\"ossler
template is identified, up to inversion symmetry, with the template of the
horseshoe mechanism. Thus R\"ossler links and horseshoe links are treated here
as belonging to the same template class. We therefore use the terms
\emph{R\"ossler link} and \emph{horseshoe link} interchangeably in this sense.

A natural problem is to compare the periodic orbit types of these systems. In
particular, one may ask which Lorenz links can also be embedded in the
horseshoe template, and conversely which horseshoe links can be realized on
the Lorenz template. In his lecture notes, Kofman conjectured that every
horseshoe link, equivalently every R\"ossler link in the setting considered
here, is a Lorenz link \cite{KofmanHorseshoeBeamer}. Thus the comparison
between Lorenz links and horseshoe/R\"ossler links is not only a question of
separating the two families, but also a question of understanding the possible
overlap between them.

In this paper, we show that the relationship between these two families has
two complementary sides. On one hand, the overlap is nontrivial: we prove
Kofman's conjecture for an infinite family of horseshoe/R\"ossler links by showing
that these links are also \(V\)-links and hence Lorenz links. On the other
hand, the two families are far from being the same. Extending the classical
obstruction of Holmes and Williams for torus knots
\cite{HolmesWilliams1985}, we construct large classes of Lorenz links,
including torus links, hyperbolic links, and satellite links, which cannot be
embedded in the horseshoe template, and hence cannot be realized as R\"ossler
links.

Our approach is braid-theoretic. Lorenz links admit positive minimal braid
representatives called \(V\)-braids \cite{de2024lorenz}. We prove the
corresponding result for horseshoe links: their positive minimal braid
representatives are precisely \(W\)-braids. Thus the problem of comparing
Lorenz links with horseshoe links can be studied through the comparison
between \(V\)-braids and \(W\)-braids.

We now recall the notation for \(V\)-braids. Let \(p,q\) be positive integers
with \(2\leq p\leq q\). Let
\[
2\leq u_1<\cdots<u_m\leq p
\qquad\text{and}\qquad
2\leq r_1<\cdots<r_n<p
\]
be possibly empty sequences, and let
\[
v_1,\dots,v_m,s_1,\dots,s_n
\]
be positive integers. The associated \(V\)-braid is the positive braid on
\(p\) strands
\[
\begin{aligned}
\beta_V={}&
(\sigma_{p-1}\sigma_{p-2}\cdots\sigma_{p-u_1+1})^{v_1}
\cdots
(\sigma_{p-1}\sigma_{p-2}\cdots\sigma_{p-u_m+1})^{v_m} \\
&\cdot
(\sigma_1\sigma_2\cdots\sigma_{r_1-1})^{s_1}
\cdots
(\sigma_1\sigma_2\cdots\sigma_{r_n-1})^{s_n}
(\sigma_1\sigma_2\cdots\sigma_{p-1})^q.
\end{aligned}
\]
If one of the sequences is empty, the corresponding product is omitted. Since
\(q\geq p\), the final torus block contains at least one full twist on the
\(p\) strands.

The closure of \(\beta_V\) is denoted by
\[
V((u_1,\overline{v_1}),\dots,(u_m,\overline{v_m}),
(r_1,s_1),\dots,(r_n,s_n),(p,q)).
\]
The bar over \(v_i\) indicates that the corresponding block is a decreasing
block, whereas the unbarred pairs \((r_j,s_j)\) denote increasing blocks.

We now introduce the class of \(W\)-braids, which will play the role of
minimal braid representatives for horseshoe links. Let \(m\geq 2\), and let
\(s_2,\dots,s_m\) be positive integers with
\[
s_m\geq m.
\]
Let
\[
m\geq w_1>\cdots>w_k\geq 2
\]
be a possibly empty decreasing sequence. If this sequence is nonempty, assume
also that
\[
s_m>w_1.
\]
A \(W\)-braid is a positive braid on \(m\) strands of the form
\[
\begin{aligned}
\beta_W={}&
(\sigma_1\sigma_2\cdots\sigma_{w_1-1})
\cdots
(\sigma_1\sigma_2\cdots\sigma_{w_k-1}) \\
&\cdot
(\sigma_1)^{s_2}
(\sigma_2\sigma_1)^{s_3}
\cdots
(\sigma_{m-1}\sigma_{m-2}\cdots\sigma_1)^{s_m}.
\end{aligned}
\]
If the sequence \(w_1,\dots,w_k\) is empty, the initial product is omitted.
Since \(s_m\geq m\), the final decreasing block contains at least one full
twist on the \(m\) strands.

The closure of \(\beta_W\) is called a \(W\)-link and is denoted by
\[
W((w_1,1),\dots,(w_k,1),
(2,\overline{s_2}),\dots,(m,\overline{s_m})).
\]
The initial pairs \((w_i,1)\) denote increasing blocks with exponent one,
whereas the barred pairs \((j,\overline{s_j})\) denote decreasing blocks
\[
(\sigma_{j-1}\sigma_{j-2}\cdots\sigma_1)^{s_j}.
\]

Our first structural result shows that \(W\)-braids give exactly the minimal
braid representatives of horseshoe links.

\begin{theorem}\label{thm:intro-minimal-W}
A minimal braid representative of a horseshoe link is a \(W\)-braid.
Conversely, every \(W\)-braid is a minimal braid representative of a
horseshoe link.
\end{theorem}

The relationship between \(V\)-links and \(W\)-links is not simply one of
disjointness. Our first comparison result shows that the two families have a
natural infinite intersection. This is motivated by a conjecture of Kofman
that every horseshoe link should be a Lorenz link. The result below verifies Kofman's conjecture
for an infinite family of \(W\)-links by showing that these links can be
rewritten as \(V\)-links, and hence are Lorenz links.

\begin{theorem}\label{thm:intro-V-W-intersection}
Let \(m\geq 2\), and let \(s_2,\dots,s_m\) be positive integers with
\(s_m\geq m\). Then the \(W\)-link
\[
W((2,\overline{2s_2}), (3,\overline{3s_3}), \dots,
(m-1,\overline{(m-1)s_{m-1}}), (m,\overline{s_m}))
\]
is equivalent to the \(V\)-link
\[
V((2,\overline{2s_2}), (3,\overline{3s_3}), \dots,
(m-1,\overline{(m-1)s_{m-1}}),(m,s_m)).
\]
Here the pairs indexed by \(2,\dots,m-1\) are omitted when \(m=2\).
\end{theorem}

Thus there are infinitely many links which can be embedded in both the Lorenz
template and the horseshoe template. This shows that the comparison between
Lorenz links and horseshoe links is subtler than simply separating two
disjoint classes.

On the other hand, Holmes and Williams proved that the two classes do not
coincide even among torus knots: if the torus knot \(T(p,q)\), with \(p<q\),
is a horseshoe knot, then
\[
3p\leq 2q.
\]
Thus many Lorenz torus knots cannot be realized as horseshoe knots.

In this paper, we show that this difference is much larger than the
Holmes--Williams obstruction suggests. We extend this phenomenon from torus
knots to torus links, and we construct infinitely many hyperbolic Lorenz knots
and links, as well as infinitely many satellite Lorenz knots and links, which
cannot be embedded in the horseshoe template.

Our main tool for detecting \(V\)-links which are not \(W\)-links is the
following crossing-number obstruction.

\begin{theorem}\label{thm:intro-crossing-obstruction}
Let \(p,q\) be integers with \(2\leq p\leq q\), and let
\[
V=
V((u_1,\overline{v_1}),\dots,(u_m,\overline{v_m}),
(r_1,s_1),\dots,(r_n,s_n),(p,q))
\]
be a \(V\)-link, where the sequences \(u_1,\dots,u_m\) and
\(r_1,\dots,r_n\) are allowed to be empty. Whenever they are nonempty, assume
that
\[
2\leq u_1<\cdots<u_m\leq p
\]
and
\[
2\leq r_1<\cdots<r_n<p.
\]
Set
\[
E(V)=
(q-p)(p-1)
+\sum_{i=1}^{m}v_i(u_i-1)
+\sum_{j=1}^{n}s_j(r_j-1).
\]
If
\[
E(V)<\frac{(p-1)(p-2)}{2},
\]
then \(V\) is not a \(W\)-link. In particular, \(V\) cannot be embedded in the
horseshoe template.

Moreover, if \(V\) is a knot and
\[
E(V)<\frac{p(p-1)}{2},
\]
then \(V\) is not a \(W\)-link.
\end{theorem}

This obstruction is the main mechanism used in the paper to construct Lorenz
links which cannot be embedded in the horseshoe template. It allows us to
extend the Holmes--Williams phenomenon beyond torus knots. In particular, we
obtain multi-component Lorenz torus links, hyperbolic Lorenz knots and links,
and satellite Lorenz knots and links which are not horseshoe links.

We now state the main geometric applications.

\begin{theorem}\label{thm:intro-torus-link-obstruction}
Let \(2\leq p<q\). If the torus link \(T(p,q)\) is a \(W\)-link, then
\[
2q\geq 3p-2.
\]
If \(T(p,q)\) is a knot and is a \(W\)-link, then
\[
2q\geq 3p.
\]
\end{theorem}

\begin{theorem}\label{thm:intro-hyperbolic-knots}
There are infinitely many hyperbolic Lorenz knots which are not horseshoe
links. Consequently, there are infinitely many hyperbolic Lorenz knots which
are not R\"ossler links.
\end{theorem}

\begin{theorem}\label{thm:intro-hyperbolic-links}
There are infinitely many hyperbolic Lorenz links with more than one component
which are not horseshoe links. Consequently, there are infinitely many
hyperbolic Lorenz links with more than one component which are not R\"ossler
links.
\end{theorem}

\begin{theorem}\label{thm:intro-satellite}
There are infinitely many satellite Lorenz knots and links which are not
horseshoe links. Consequently, there are infinitely many satellite Lorenz knots
and links which are not R\"ossler links.
\end{theorem}

Finally, the obstruction is inherited by superlinks. Indeed, if a link can be
embedded in the horseshoe template, then every sublink can also be embedded in
the same template. Therefore, if a Lorenz link contains a sublink which is not
a horseshoe link, then the whole Lorenz link is not a horseshoe link. This
shows that our examples generate many more Lorenz links which are not
R\"ossler links.

These results show that the difference between Lorenz links and
horseshoe/R\"ossler links is substantially larger than the classical
torus-knot obstruction suggests. They also illustrate how braid-theoretic
invariants can be used to distinguish templates through the link types of
their periodic orbits.

The paper is organized as follows. In
Section~\ref{sec:horseshoe-braids}, we recall the braid description of links
carried by the horseshoe template and relate horseshoe links to the braid
families used later in the paper. In
Section~\ref{sec:minimal-braids-horseshoe-links}, we introduce \(W\)-braids
and prove that they are precisely the minimal braid representatives of
horseshoe links. In Section~\ref{sec:V-links-that-are-W-links}, we exhibit an
infinite family of links which are both \(V\)-links and \(W\)-links. Finally,
in Section~\ref{sec:crossing-obstruction-V-W}, we prove the crossing-number
obstruction and apply it to construct torus, hyperbolic, and satellite Lorenz
links which cannot be embedded in the horseshoe template, and hence are not
R\"ossler links.

\subsection*{Acknowledgements}
The authors would like to thank Tali Pinsky and Andr{\'e} de Carvalho for helpful discussions.

\begin{figure}
\includegraphics[scale=0.15]{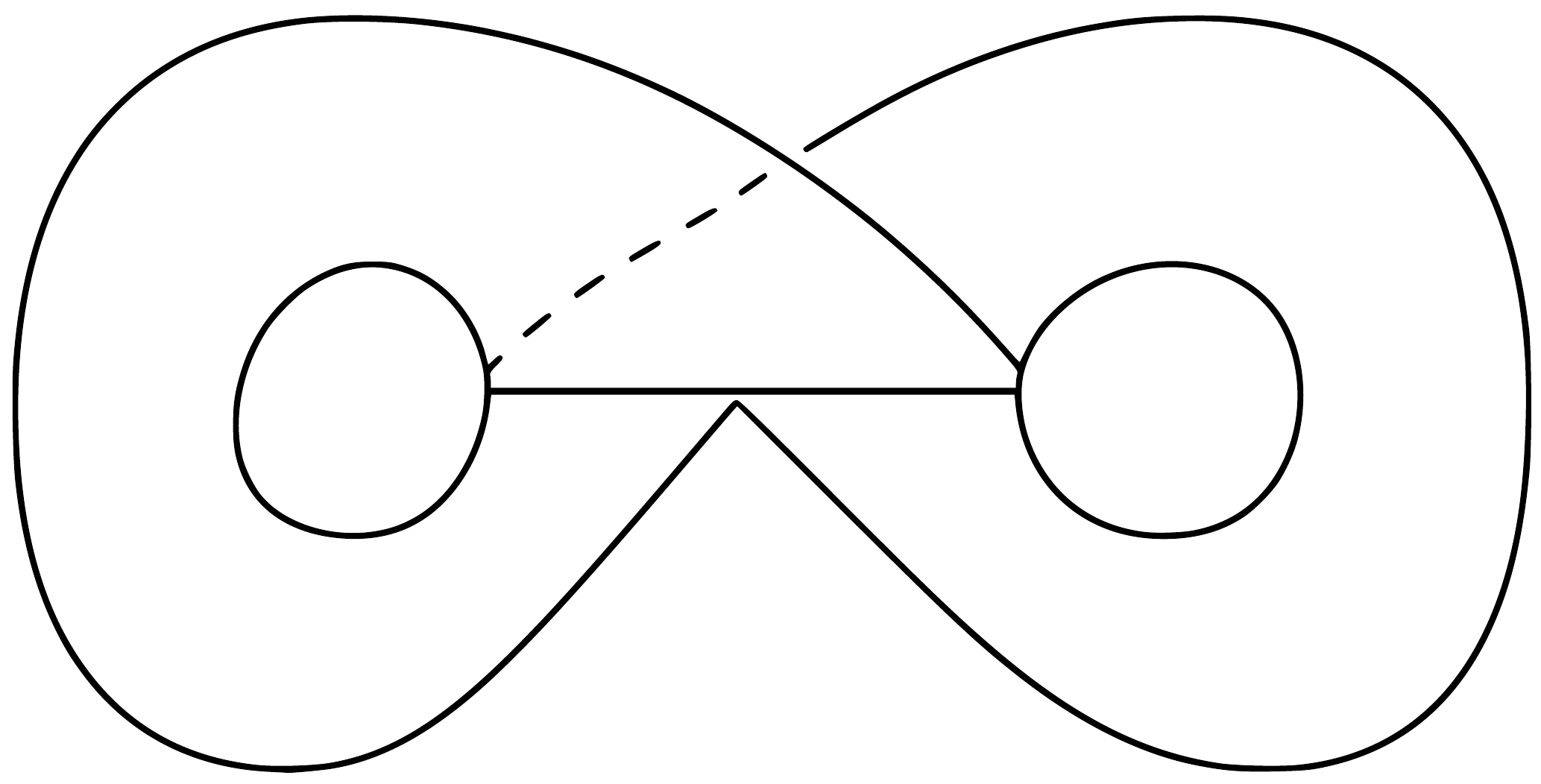} 
  \caption{Lorenz Template}
  \label{LT}
\end{figure}

\begin{figure}
\includegraphics[scale=0.15]{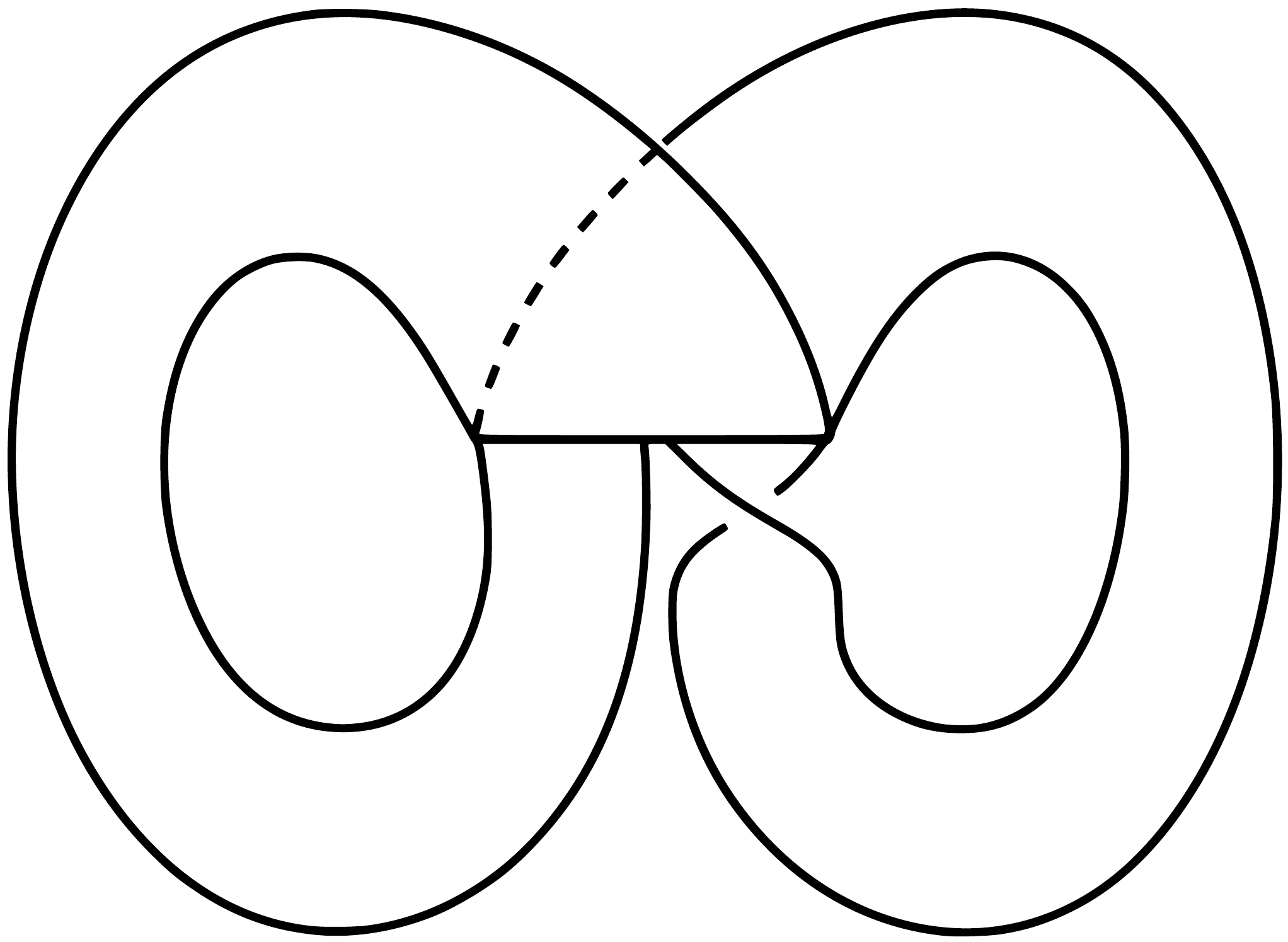} 
  \caption{Horseshoe Template}
  \label{HT}
\end{figure} 
 
\section{From the horseshoe template to horseshoe braids}
\label{sec:horseshoe-braids}

In this section we recall a braid description of links carried by the
horseshoe template. Our goal is to pass from the template definition of a
horseshoe link to a concrete family of positive braids that we shall use
throughout the paper.

We use the convention that boundary-parallel trivial components of the
template are ignored. More precisely, if a link carried by the horseshoe
template contains split unknotted components parallel to a boundary component
of the template, we discard these components. These components correspond to
the parameter \(d\) in the \(T^1\)-link description below, and they do not play
a role in the questions considered in this paper.

We identify the horseshoe template with the Lorenz-like template
\(\mathcal L(0,1)\). We now recall the corresponding \(T^1\)-link notation.

\begin{definition}
Let \(k\geq 0\). Let \(d\geq 0\), let \(r_{k+1}\geq 1\), and, if
\(k\geq 1\), let
\[
2\leq r_1<\cdots<r_k\leq r_{k+1}
\]
and let \(s_1,\dots,s_k\) be positive integers. The \emph{\(T^1\)-link}
\[
T^1((r_1,s_1),\ldots,(r_k,s_k),(r_{k+1};d))
\]
is defined to be the closure of the braid
\[
(\sigma_{1+d}\sigma_{2+d}\cdots\sigma_{d+r_1-1})^{s_1}
\cdots
(\sigma_{1+d}\sigma_{2+d}\cdots\sigma_{d+r_k-1})^{s_k}
\Delta_{r_{k+1},d+r_{k+1}},
\]
where \(\Delta_{r_{k+1},d+r_{k+1}}\) denotes a positive half twist on the last
\(r_{k+1}\) strands of the trivial braid with \(d+r_{k+1}\) strands.

If \(k=0\), we define
\[
T^1((r_1;d))
\]
to be the closure of
\[
\Delta_{r_1,d+r_1}.
\]
In particular, \(T^1((1;d))\) is the unlink with \(d+1\) components represented
by the trivial braid with \(d+1\) strands.
\end{definition}

The following theorem gives the braid description of links carried by the
horseshoe template. It is the case \(n=1\) of \cite[Theorem~4.5]{Volumebounds}.

\begin{theorem}\label{thm:horseshoe-T1-links}
A link can be embedded in the horseshoe template if and only if it is a
\(T^1\)-link. Equivalently, every link embedded in the horseshoe template is a
\(T^1\)-link, and every \(T^1\)-link can be embedded in the horseshoe template.
\end{theorem}

\begin{proof}
The horseshoe template is identified with the Lorenz-like template
\(\mathcal L(0,1)\). The result follows from \cite[Theorem~4.5]{Volumebounds}, by
taking \(n=1\).
\end{proof}

We now introduce the explicit braid family used in this paper.

\begin{definition}
Let \(m\geq 2\), and let \(s_2,\dots,s_m\) be positive integers. The link
\[
S((2,\overline{s_2}), (3,\overline{s_3}), \dots,
(m-1,\overline{s_{m-1}}), (m,\overline{s_m}))
\]
is defined to be the closure of the braid
\[
(\sigma_1)^{s_2}
(\sigma_2\sigma_1)^{s_3}
\cdots
(\sigma_{m-2}\cdots\sigma_2\sigma_1)^{s_{m-1}}
(\sigma_{m-1}\cdots\sigma_2\sigma_1)^{s_m}.
\]
We call this braid an \emph{\(S\)-braid}.
\end{definition}

The following correspondence appears, in an informal form, in Kofman's lecture
notes \cite{KofmanHorseshoeBeamer}. There, the horseshoe template is related
to the Lorenz-like template \(\mathcal L(0,1)\), and the corresponding braid
description is stated in terms of reversed braid factors. For completeness, we
give a formal statement and proof in the notation used in this paper.

\begin{theorem}\label{thm:horseshoe-S-links}
Up to boundary-parallel trivial components, a link is a horseshoe link if and
only if it is of the form
\[
S((2,\overline{s_2}), (3,\overline{s_3}), \dots,
(m-1,\overline{s_{m-1}}), (m,\overline{s_m}))
\]
for some \(m\geq 2\) and positive integers \(s_2,\dots,s_m\).
\end{theorem}

\begin{proof}
Let \(L\) be a horseshoe link. By definition, \(L\) can be embedded in the
horseshoe template. By Theorem~\ref{thm:horseshoe-T1-links}, \(L\) is a
\(T^1\)-link.

By our convention, we ignore boundary-parallel trivial components. Thus we may
assume that the \(T^1\)-representative has no extra trivial strands, that is,
\(d=0\). Therefore \(L\) is represented by a braid of the form
\[
(\sigma_1\cdots\sigma_{r_1-1})^{s_1}
\cdots
(\sigma_1\cdots\sigma_{r_k-1})^{s_k}
\Delta_m,
\]
where \(m=r_{k+1}\).

For notational convenience, for each \(2\leq j\leq m\), let \(a_j\) be the
exponent of the block \((\sigma_1\sigma_2\cdots\sigma_{j-1})\), and set
\(a_j=0\) if no such block occurs. Then the braid can be written uniformly as
\[
(\sigma_1)^{a_2}
(\sigma_1\sigma_2)^{a_3}
\cdots
(\sigma_1\sigma_2\cdots\sigma_{m-1})^{a_m}
\Delta_m,
\]
where \(a_2,\dots,a_m\) are non-negative integers.

Set
\[
P_j=\sigma_1\sigma_2\cdots\sigma_{j-1}
\qquad\text{and}\qquad
A_{j-1}=\sigma_{j-1}\sigma_{j-2}\cdots\sigma_1.
\]
We write the positive half twist on \(j\) strands as
\[
\Delta_j=A_1A_2\cdots A_{j-1}.
\]
We shall use the standard braid identity
\[
P_j\Delta_j=\Delta_j A_{j-1}.
\]
Hence, for every \(a\geq 0\),
\[
P_j^a\Delta_j=\Delta_j A_{j-1}^a.
\]

Now consider the braid
\[
P_2^{a_2}P_3^{a_3}\cdots P_m^{a_m}\Delta_m.
\]
Since
\[
\Delta_m=\Delta_{m-1}A_{m-1},
\]
we have
\[
P_m^{a_m}\Delta_m
=
\Delta_m A_{m-1}^{a_m}
=
\Delta_{m-1}A_{m-1}^{a_m+1}.
\]
Therefore
\[
P_2^{a_2}\cdots P_m^{a_m}\Delta_m
=
P_2^{a_2}\cdots P_{m-1}^{a_{m-1}}
\Delta_{m-1}A_{m-1}^{a_m+1}.
\]
Applying the same identity to
\(P_{m-1}^{a_{m-1}}\Delta_{m-1}\), we obtain
\[
P_{m-1}^{a_{m-1}}\Delta_{m-1}
=
\Delta_{m-1}A_{m-2}^{a_{m-1}}
=
\Delta_{m-2}A_{m-2}^{a_{m-1}+1}.
\]
Hence
\[
P_2^{a_2}\cdots P_m^{a_m}\Delta_m
=
P_2^{a_2}\cdots P_{m-2}^{a_{m-2}}
\Delta_{m-2}
A_{m-2}^{a_{m-1}+1}
A_{m-1}^{a_m+1}.
\]
Iterating this argument gives
\[
P_2^{a_2}P_3^{a_3}\cdots P_m^{a_m}\Delta_m
=
A_1^{a_2+1}A_2^{a_3+1}\cdots A_{m-1}^{a_m+1}.
\]
Thus the closure of
\[
(\sigma_1)^{a_2}
(\sigma_1\sigma_2)^{a_3}
\cdots
(\sigma_1\sigma_2\cdots\sigma_{m-1})^{a_m}
\Delta_m
\]
is the closure of
\[
A_1^{a_2+1}A_2^{a_3+1}\cdots A_{m-1}^{a_m+1},
\]
that is, the closure of the braid
\[
(\sigma_1)^{a_2+1}
(\sigma_2\sigma_1)^{a_3+1}
\cdots
(\sigma_{m-1}\cdots\sigma_2\sigma_1)^{a_m+1}.
\]
This is an \(S\)-braid. Hence \(L\) is an \(S\)-link.

Conversely, suppose \(L\) is the closure of an \(S\)-braid
\[
(\sigma_1)^{s_2}
(\sigma_2\sigma_1)^{s_3}
\cdots
(\sigma_{m-1}\cdots\sigma_2\sigma_1)^{s_m}.
\]
Since each \(s_j\) is positive, write
\[
s_j=a_j+1
\qquad
(2\leq j\leq m),
\]
where \(a_j\geq 0\). Reversing the braid manipulation above, the closure of
this braid is equivalent to the closure of
\[
(\sigma_1)^{a_2}
(\sigma_1\sigma_2)^{a_3}
\cdots
(\sigma_1\sigma_2\cdots\sigma_{m-1})^{a_m}
\Delta_m.
\]
After omitting the zero powers, this is a \(T^1\)-braid with \(d=0\).
Therefore, by Theorem~\ref{thm:horseshoe-T1-links}, \(L\) can be embedded in
the horseshoe template. Hence \(L\) is a horseshoe link.

Thus, up to boundary-parallel trivial components, the links carried by the
horseshoe template are exactly the links represented by \(S\)-braids.
\end{proof}
 
\begin{remark}\label{rem:boundary-parallel-components}
Throughout this paper, we disregard components of a horseshoe link that are
parallel to boundary components of the horseshoe template. These components
are split trivial components. In the \(T^1\)-link description, they correspond
to the parameter \(d\). Thus, if a horseshoe link contains such components,
then it is the split union of its nontrivial part with a finite number of
unknotted components. Conversely, such split unknotted components may be added
without changing the part of the link carried by the nontrivial dynamics of
the template.

This convention does not affect the arguments in this paper. Indeed, if a
link is carried by the horseshoe template, then every sublink obtained by
deleting boundary-parallel trivial components is also carried by the
horseshoe template. Therefore, the obstructions used below apply to the
nontrivial part of the link. In the extreme case where all components are
boundary-parallel, the link is simply an unlink, and this case is completely
understood. Hence we ignore boundary-parallel trivial components in what
follows.
\end{remark}

\section{Minimal braids of horseshoe links}
\label{sec:minimal-braids-horseshoe-links}

In this section we describe minimal braid representatives of horseshoe links.
As in Remark~\ref{rem:boundary-parallel-components}, we ignore
boundary-parallel trivial components. By
Theorem~\ref{thm:horseshoe-S-links}, after deleting such components, every
horseshoe link is represented by an \(S\)-braid. The main result is that every
such link admits a minimal braid representative in a standard form which we
call a \(W\)-braid. Conversely, every \(W\)-braid is a minimal braid
representative of a horseshoe link.

We recall the terminology used in this section. If \(\beta\in B_p\), its
closure will be denoted by \(\widehat{\beta}\). A link \(L\subset S^3\) is
said to be represented by the closed \(p\)-braid \(\widehat{\beta}\) if
\(\widehat{\beta}\) is equivalent to \(L\), that is, ambient isotopic to \(L\)
in \(S^3\). The \emph{braid index} \(b(L)\) of \(L\) is the smallest integer
\(p\) such that \(L\) can be represented by a closed \(p\)-braid. A braid
representative \(\beta\in B_p\) of \(L\) is called \emph{minimal} if
\(p=b(L)\).

We use the following notation. For \(j\geq 1\), set
\[
A_j=\sigma_j\sigma_{j-1}\cdots\sigma_1.
\]
Thus the \(S\)-braid
\[
S((2,\overline{s_2}), (3,\overline{s_3}),\dots,
(m,\overline{s_m}))
\]
is
\[
A_1^{s_2}A_2^{s_3}\cdots A_{m-1}^{s_m}.
\]
In the proof below, intermediate braids may contain zero powers \(A_j^0\);
these factors are always omitted.

Let \(B\in B_p\). We say that \(B\) has \emph{span contained in} \([1,q]\),
and write
\[
\operatorname{span}(B)\subset [1,q],
\]
if \(B\) lies in the standard embedded subgroup of \(B_p\) generated by
\[
\sigma_1,\dots,\sigma_{q-1}.
\]
Equivalently, \(B\) admits a braid word involving only the generators
\(\sigma_1^{\pm1},\dots,\sigma_{q-1}^{\pm1}\).

\begin{definition}
Let \(m\geq 2\), let \(s_2,\dots,s_m\) be positive integers, and let
\(w_1,\dots,w_k\) be a possibly empty sequence of positive integers satisfying
\[
m\geq w_1>\cdots>w_k\geq 2
\]
whenever \(k>0\). Assume also that
\[
s_m\geq m,
\]
and, if \(k>0\), that
\[
s_m>w_1.
\]
A \(W\)-braid is a braid of the form
\[
(\sigma_1\sigma_2\cdots\sigma_{w_1-1})
\cdots
(\sigma_1\sigma_2\cdots\sigma_{w_k-1})
A_1^{s_2}A_2^{s_3}\cdots A_{m-1}^{s_m}.
\]
Its closure is called a \(W\)-link.

Equivalently, we denote such a braid, or its closure, by
\[
W((w_1,1),\dots,(w_k,1),
(2,\overline{s_2}),\dots,(m,\overline{s_m})).
\]
If \(k=0\), the initial sequence is omitted.

We also regard the trivial one-strand braid as a degenerate \(W\)-braid. Its
closure is the unknot.
\end{definition}

The condition \(s_m>w_1\), when the initial sequence is nonempty, is included
so that \(W\)-braids form the precise standard class corresponding to minimal
braid representatives of horseshoe links. Indeed, in the proof of
Theorem~\ref{thm:minimal-horseshoe-W}, every \(S\)-braid is reduced to a
\(W\)-braid by successively removing strands from the right. If an initial
\(W\)-block of width \(w_1\) is created, then the final exponent in the
remaining tail is of the form
\[
t_i=s_i+t_{i+1},
\]
where \(w_1=t_{i+1}\). Since \(s_i>0\), the resulting \(W\)-braid automatically
satisfies
\[
t_i>w_1.
\]
Thus the condition \(s_m>w_1\) is forced by the construction from
\(S\)-braids.

Conversely, this condition guarantees that, when a \(W\)-braid is converted
back into an \(S\)-braid by reversing the construction, all exponents in the
resulting \(S\)-braid are positive. If the condition were omitted, the inverse
construction could produce a braid with a zero exponent in the \(S\)-braid
tail. Such a braid would not be an \(S\)-braid in the standard sense used in
Theorem~\ref{thm:horseshoe-S-links}. Hence the condition is imposed to obtain
the desired equivalence between \(W\)-braids and minimal braid representatives
of horseshoe links.

\begin{lemma}\label{lem:W-braids-are-minimal}
\(W\)-braids are minimal braid representatives of their closures.
\end{lemma}

\begin{proof}
The degenerate one-strand \(W\)-braid is minimal, since its closure is the
unknot.

Now let \(W\) be a non-degenerate \(W\)-braid on \(m\geq 2\) strands. By
definition, \(W\) is positive and has final factor
\[
A_{m-1}^{s_m}=(\sigma_{m-1}\sigma_{m-2}\cdots\sigma_1)^{s_m},
\]
with \(s_m\geq m\). In particular, it contains
\[
A_{m-1}^m=(\sigma_{m-1}\sigma_{m-2}\cdots\sigma_1)^m,
\]
which is a positive full twist on all \(m\) strands. Therefore, by
Franks--Williams \cite[Corollary~2.4]{Franks}, the braid realizes the braid
index of its closure. Hence every \(W\)-braid is a minimal braid
representative.
\end{proof}

We shall use the following elementary isotopy.

\begin{proposition}\label{prop:horseshoe-isotopy-one-step}
Let \(\beta\in B_p\) be a braid whose span is contained in \([1,p-1]\), and
let \(q\) be an integer such that
\[
1<q\leq p-1.
\]
Then the closure of the \(p\)-braid
\[
\beta(\sigma_{p-1}\sigma_{p-2}\cdots \sigma_1)^q
\]
is isotopic to the closure of the \((p-1)\)-braid
\[
\beta
(\sigma_{p-q}\sigma_{p-q+1}\cdots \sigma_{p-2})
(\sigma_{p-2}\sigma_{p-3}\cdots \sigma_1)^q,
\]
where \(\beta\) is regarded as a braid on the first \(p-1\) strands.
\end{proposition}

\begin{proof}
Since \(\operatorname{span}(\beta)\subset [1,p-1]\), the \(p\)-th strand of
\(\beta\) is trivial. Thus the braid
\[
\beta(\sigma_{p-1}\sigma_{p-2}\cdots\sigma_1)^q
\]
may be viewed as a braid obtained from the nontrivial part of \(\beta\) by
adding one extra strand on the right.

Under the factor
\[
(\sigma_{p-1}\sigma_{p-2}\cdots\sigma_1)^q,
\]
this extra strand starts in the \(p\)-th position and is carried through the
block until it reaches the \((p-q)\)-th position. In the closure, this strand
goes once around the braid axis and can be pushed across the closure and
shrunk. This removes the \(p\)-th strand and reduces the last block from
\[
(\sigma_{p-1}\sigma_{p-2}\cdots\sigma_1)^q
\]
to
\[
(\sigma_{p-2}\sigma_{p-3}\cdots\sigma_1)^q.
\]
During this isotopy, the path of the shrunk strand contributes the additional
factor
\[
\sigma_{p-q}\sigma_{p-q+1}\cdots\sigma_{p-2}.
\]
Therefore the closure of
\[
\beta(\sigma_{p-1}\sigma_{p-2}\cdots\sigma_1)^q
\]
is isotopic to the closure of the \((p-1)\)-braid
\[
\beta
(\sigma_{p-q}\sigma_{p-q+1}\cdots\sigma_{p-2})
(\sigma_{p-2}\sigma_{p-3}\cdots\sigma_1)^q,
\]
where \(\beta\) is now regarded as a braid on the first \(p-1\) strands.
This proves the proposition.
\end{proof}

\begin{remark}\label{rem:horseshoe-isotopy-special-case}
We shall mainly use Proposition~\ref{prop:horseshoe-isotopy-one-step} in the
following form. If \(p=r+1\), then the closure of the \((r+1)\)-braid
\[
\beta A_r^q
\]
is isotopic to the closure of the \(r\)-braid
\[
\beta
(\sigma_{r-q+1}\sigma_{r-q+2}\cdots\sigma_{r-1})
A_{r-1}^q.
\]
By cyclically permuting the braid closure, the factor
\[
\sigma_{r-q+1}\sigma_{r-q+2}\cdots\sigma_{r-1}
\]
may be pulled once around the closure and placed on the leftmost \(q\) strands.
Thus the same link is also represented by the closure of
\[
(\sigma_1\sigma_2\cdots\sigma_{q-1})\beta A_{r-1}^q.
\]
\end{remark}

\begin{theorem}\label{thm:minimal-horseshoe-W}
After deleting boundary-parallel trivial components, every horseshoe link
admits a minimal braid representative which is a \(W\)-braid. Conversely,
every \(W\)-braid is a minimal braid representative of a horseshoe link.
\end{theorem}

\begin{proof}
Let \(L\) be a horseshoe link, and let \(L'\) be the link obtained from \(L\)
after deleting boundary-parallel trivial components. By
Theorem~\ref{thm:horseshoe-S-links}, \(L'\) is the closure of an \(S\)-braid
\[
H_m=A_1^{s_2}A_2^{s_3}\cdots A_{m-1}^{s_m}
\]
on \(m\) strands. We show that the closure of \(H_m\) can be represented by a
\(W\)-braid.

If \(s_m\geq m\), then \(H_m\) is already a \(W\)-braid, with empty initial
sequence. Hence, by Lemma~\ref{lem:W-braids-are-minimal}, it is a minimal
braid representative of \(L'\).

Assume now that \(s_m<m\). If \(m=2\) and \(s_2=1\), then the closure of
\(H_m=A_1\) is the unknot. It is represented by the trivial one-strand braid,
which is a degenerate \(W\)-braid. Hence the conclusion holds in this case.

We may therefore assume that \(H_m\) is not this degenerate case. Define
\[
t_j=s_j+s_{j+1}+\cdots+s_m
\qquad
(2\leq j\leq m).
\]
Since we have excluded the case \(m=2\) and \(s_2=1\), we have
\[
t_2\geq 2.
\]
Hence there is a largest index \(i\) such that
\[
t_i\geq i.
\]

We now reduce the braid from the right. Suppose that, at some stage, the
current braid has the form
\[
B_j A_{j-1}^{t_j},
\]
where \(B_j\) is a \(j\)-strand braid with span contained in \([1,j-1]\), and
where \(j>i\). Since \(i\) is the largest index with \(t_i\geq i\), we have
\[
t_j<j.
\]

If \(t_j=1\), then necessarily \(j=m\) and \(s_m=1\), because all exponents
\(s_2,\dots,s_m\) are positive. In this case the last strand destabilizes, and
no initial \(W\)-block is created. Thus the closure is represented by the braid
\[
A_1^{s_2}A_2^{s_3}\cdots A_{m-2}^{s_{m-1}+1}.
\]

If \(1<t_j<j\), then Remark~\ref{rem:horseshoe-isotopy-special-case} applies
to the factor \(A_{j-1}^{t_j}\). It replaces the closure of
\[
B_jA_{j-1}^{t_j}
\]
by the closure of
\[
(\sigma_1\sigma_2\cdots\sigma_{t_j-1})B_jA_{j-2}^{t_j}.
\]
Thus one strand is removed from the \(S\)-braid tail, and one initial
\(W\)-block of width \(t_j\) is created on the left.

At each step, the previously created initial blocks have widths
\(t_{j+1},t_{j+2},\dots,t_m\), which are strictly smaller than \(t_j\). Hence
they involve only generators among
\[
\sigma_1,\dots,\sigma_{j-2},
\]
so the required span condition is preserved and the next step of the
iteration is valid.

Iterating this procedure for \(j=m,m-1,\dots,i+1\), we obtain a braid whose
closure is equal to the closure of
\[
(\sigma_1\sigma_2\cdots\sigma_{t_{i+1}-1})
(\sigma_1\sigma_2\cdots\sigma_{t_{i+2}-1})
\cdots
(\sigma_1\sigma_2\cdots\sigma_{t_m-1})
A_1^{s_2}A_2^{s_3}\cdots A_{i-2}^{s_{i-1}}A_{i-1}^{t_i},
\]
where any empty product is omitted.

Since \(i\) is the largest index satisfying \(t_i\geq i\), we have
\[
i\geq t_{i+1}>t_{i+2}>\cdots>t_m.
\]
Moreover, if the initial sequence is nonempty, then
\[
t_i=t_{i+1}+s_i>t_{i+1}.
\]
After omitting any block of width \(1\), the initial blocks have strictly
decreasing widths at least \(2\), the final \(S\)-braid tail has \(i\) strands
with last exponent \(t_i\geq i\), and the last exponent is strictly larger
than the first initial width whenever an initial block is present. Therefore
the resulting braid is a \(W\)-braid. By
Lemma~\ref{lem:W-braids-are-minimal}, this \(W\)-braid is a minimal braid
representative of \(L'\).

Conversely, let \(W\) be a \(W\)-braid. If \(W\) is the degenerate one-strand
\(W\)-braid, then its closure is the unknot, which is a horseshoe link. Thus
assume that \(W\) is non-degenerate. Then \(W\) has the form
\[
W=
(\sigma_1\sigma_2\cdots\sigma_{w_1-1})
\cdots
(\sigma_1\sigma_2\cdots\sigma_{w_k-1})
A_1^{s_2}A_2^{s_3}\cdots A_{m-1}^{s_m},
\]
where
\[
m\geq w_1>\cdots>w_k\geq 2,
\qquad
s_m\geq m,
\]
and, if \(k>0\),
\[
s_m>w_1.
\]
By Lemma~\ref{lem:W-braids-are-minimal}, \(W\) is a minimal braid
representative of its closure.

It remains to show that the closure of \(W\) is a horseshoe link. If \(k=0\),
then
\[
W=A_1^{s_2}A_2^{s_3}\cdots A_{m-1}^{s_m},
\]
which is already an \(S\)-braid. Hence its closure is a horseshoe link.

We may therefore assume that \(k\geq 1\). We reverse the construction above.
Since
\[
s_m>w_1,
\]
we may split the last factor as
\[
A_{m-1}^{s_m}=A_{m-1}^{s_m-w_1}A_{m-1}^{w_1}.
\]
Write
\[
\Gamma_1=
(\sigma_1\sigma_2\cdots\sigma_{w_2-1})
\cdots
(\sigma_1\sigma_2\cdots\sigma_{w_k-1})
A_1^{s_2}A_2^{s_3}\cdots A_{m-2}^{s_{m-1}}A_{m-1}^{s_m-w_1}.
\]
Then the closure of \(W\) is the closure of
\[
(\sigma_1\sigma_2\cdots\sigma_{w_1-1})\Gamma_1A_{m-1}^{w_1}.
\]
Applying the inverse of Remark~\ref{rem:horseshoe-isotopy-special-case}, this
closure is represented by the closure of
\[
\Gamma_1A_m^{w_1}.
\]
Thus the first initial block has been absorbed into the \(S\)-braid tail, and
the number of strands has increased by one.

Now repeat the same operation for the next initial block. Since
\[
w_1>w_2>\cdots>w_k,
\]
at the second step we can split
\[
A_m^{w_1}=A_m^{w_1-w_2}A_m^{w_2}
\]
and absorb the block
\[
\sigma_1\sigma_2\cdots\sigma_{w_2-1}
\]
into the tail, producing a new final factor \(A_{m+1}^{w_2}\). Continuing in
this way, each initial block is absorbed into the \(S\)-braid tail.

After finitely many steps, all initial blocks have been absorbed into the
\(S\)-braid tail. If \(k=1\), the closure of \(W\) is represented by
\[
A_1^{s_2}\cdots A_{m-2}^{s_{m-1}}
A_{m-1}^{s_m-w_1}
A_m^{w_1}.
\]
If \(k\geq 2\), the closure of \(W\) is represented by
\[
A_1^{s_2}\cdots A_{m-2}^{s_{m-1}}
A_{m-1}^{s_m-w_1}
A_m^{w_1-w_2}
\cdots
A_{m+k-2}^{w_{k-1}-w_k}
A_{m+k-1}^{w_k}.
\]
In both cases, all exponents are positive. Hence the resulting braid is an
\(S\)-braid. Therefore, by Theorem~\ref{thm:horseshoe-S-links}, its closure is
a horseshoe link.

Thus, after deleting boundary-parallel trivial components, every horseshoe
link admits a minimal braid representative which is a \(W\)-braid, and every
\(W\)-braid is a minimal braid representative of a horseshoe link.
\end{proof}

\section{Some \(V\)-links that are \(W\)-links}
\label{sec:V-links-that-are-W-links}

The relationship between \(V\)-links and \(W\)-links is not simply one of
disjointness. Our first comparison result shows that the two families have a
natural infinite intersection. This is motivated by a conjecture of Kofman
that every horseshoe link should be a Lorenz link
\cite{KofmanHorseshoeBeamer}. In the setting considered here, the same
interpretation applies to R\"ossler links, since they are identified with
horseshoe links. The result below verifies Kofman's conjecture for a concrete
infinite family of \(W\)-links: these links can be rewritten as \(V\)-links,
and hence are Lorenz links.

For each \(j\geq 2\), the increasing and decreasing standard positive roots
have the same \(j\)-th power:
\[
(\sigma_1\sigma_2\cdots\sigma_{j-1})^j
=
(\sigma_{j-1}\sigma_{j-2}\cdots\sigma_1)^j.
\]
Indeed, both braids are the positive full twist on the first \(j\) strands.
Therefore, whenever the exponent of an intermediate block is a multiple of
the number of strands involved, an increasing block appearing in a
\(V\)-braid can be replaced by a decreasing block. This produces a
\(W\)-braid with the same closure.

\begin{theorem}\label{thm:W-multiple-exponents-is-V}
Let \(m\geq 2\), and let \(s_2,\dots,s_m\) be positive integers with
\(s_m\geq m\). Then the \(W\)-link
\[
W((2,\overline{2s_2}), (3,\overline{3s_3}), \dots,
(m-1,\overline{(m-1)s_{m-1}}), (m,\overline{s_m}))
\]
is equivalent to the \(V\)-link
\[
V((2,\overline{2s_2}), (3,\overline{3s_3}), \dots,
(m-1,\overline{(m-1)s_{m-1}}),(m,s_m)).
\]
Here the pairs indexed by \(2,\dots,m-1\) are omitted when \(m=2\). Thus, for
\(m=2\), the statement says that
\[
W((2,\overline{s_2}))=V((2,s_2)).
\]
\end{theorem}

\begin{proof}
For \(j\geq 2\), set
\[
P_j=\sigma_1\sigma_2\cdots\sigma_{j-1}
\qquad\text{and}\qquad
A_{j-1}=\sigma_{j-1}\sigma_{j-2}\cdots\sigma_1.
\]
Products below are ordered with increasing \(j\) from left to right.

The \(W\)-link in the statement is the closure of the braid
\[
\left(\prod_{j=2}^{m-1} A_{j-1}^{j s_j}\right) A_{m-1}^{s_m}.
\]
When \(m=2\), the product is empty, and the braid is simply \(A_1^{s_2}\).

Now consider the \(V\)-link
\[
V((2,\overline{2s_2}), (3,\overline{3s_3}), \dots,
(m-1,\overline{(m-1)s_{m-1}}),(m,s_m)).
\]
By \cite[Lemma~3.2]{de2024lorenz}, this \(V\)-link is represented by the
closure of the braid
\[
\left(\prod_{j=2}^{m-1} P_j^{j s_j}\right) A_{m-1}^{s_m}.
\]
For each \(j=2,\dots,m-1\), we have
\[
P_j^j=A_{j-1}^j,
\]
because both sides are the positive full twist on the first \(j\) strands.
Hence
\[
P_j^{j s_j}=A_{j-1}^{j s_j}.
\]
Applying this identity for every \(j=2,\dots,m-1\), the braid representing the
\(V\)-link becomes
\[
\left(\prod_{j=2}^{m-1} A_{j-1}^{j s_j}\right) A_{m-1}^{s_m}.
\]
This is exactly the defining braid of the \(W\)-link in the statement.
Therefore the two links are equivalent.
\end{proof}

Theorem~\ref{thm:W-multiple-exponents-is-V} gives a simple infinite family of
\(V\)-links which are also \(W\)-links. In the next section, we show that this
phenomenon is far from universal: a crossing-number obstruction proves that
many \(V\)-links cannot be represented by \(W\)-braids.

\section{Crossing-number obstructions and applications}
\label{sec:crossing-obstruction-V-W}

In this section we prove a crossing-number obstruction showing that many
\(V\)-links cannot be \(W\)-links. The main idea is that \(W\)-braids have a
large number of forced crossings, while many \(V\)-braids have comparatively
few crossings. Since both families are represented by positive minimal braids,
the number of crossings becomes an obstruction.

Before proving the obstruction, we recall that horseshoe links themselves are
not restricted to a single geometric type. Torus links, hyperbolic links, and
satellite links all occur among horseshoe links. For example, for \(q\geq 2\),
the \(W\)-link
\[
W((2,\overline{q}))
\]
is the torus link \(T(2,q)\). Hyperbolic horseshoe knots also occur; see, for
instance, \cite[Theorem~1.2]{twofulltwists} and
\cite[Theorem~7.10]{de2024lorenz}. Satellite examples occur as well; they can be constructed by applying the same
idea as in \cite[Theorem~6.3]{de2024lorenz}. 
Thus the obstruction proved below is not a
restriction on the possible geometric types of horseshoe links. Rather, it
detects specific Lorenz links of various geometric types which cannot be
realized in the horseshoe template.

\subsection{The obstruction}

We begin with an elementary consequence of Stallings' theorem.

\begin{lemma}\label{lem:crossing-number-positive-same-strands}
Let \(\beta\) and \(\beta'\) be positive braid representatives of the same link.
If \(\beta\) and \(\beta'\) have the same number of strands, then they have the
same number of crossings.
\end{lemma}

\begin{proof}
Let \(L\) be the link represented by both braids. If \(L\) is represented by a
positive braid with \(p\) strands and \(c\) crossings, then the Seifert surface
obtained from the braid has Euler characteristic
\[
\chi=p-c.
\]
By Stallings' theorem \cite{Genus}, this surface is a fiber surface and
realizes the maximal Euler characteristic among Seifert surfaces for \(L\).
Thus \(p-c\) is an invariant of the link \(L\).

Now suppose that \(\beta\) and \(\beta'\) have the same number \(p\) of
strands, and let their crossing numbers be \(c\) and \(c'\), respectively.
Since both are positive braid representatives of the same link, we have
\[
p-c=p-c'.
\]
Hence \(c=c'\).
\end{proof}

We now record a simple constraint on \(W\)-braids whose closures are knots.
This is where the condition \(s_m>w_1\), when the initial sequence is
nonempty, is used.

\begin{lemma}\label{lem:W-knot-last-exponent}
Let
\[
W((w_1,1),\dots,(w_k,1),
(2,\overline{s_2}),\dots,(m,\overline{s_m}))
\]
be a \(W\)-link, where the initial sequence \(w_1,\dots,w_k\) is allowed to be
empty. If this \(W\)-link is a knot, then
\[
s_m>m.
\]
\end{lemma}

\begin{proof}
By definition of a \(W\)-link, we have \(s_m\geq m\). Suppose, for a
contradiction, that
\[
s_m=m.
\]

Consider the permutation induced by the \(W\)-braid on its \(m\) strands. The
number of components of the closure of a braid is equal to the number of
cycles of its induced permutation.

The last factor of the \(W\)-braid is
\[
(\sigma_{m-1}\sigma_{m-2}\cdots\sigma_1)^m.
\]
Since \(\sigma_{m-1}\sigma_{m-2}\cdots\sigma_1\) induces an \(m\)-cycle, this
factor induces the identity permutation.

If the initial sequence is empty, then every remaining factor
\[
(\sigma_{j-1}\sigma_{j-2}\cdots\sigma_1)^{s_j},
\qquad 2\leq j\leq m-1,
\]
involves only the first \(m-1\) strands. Hence the induced permutation fixes
the \(m\)-th strand.

If the initial sequence is nonempty, then the defining condition \(s_m>w_1\)
and the assumption \(s_m=m\) imply
\[
w_1<m.
\]
Thus every initial factor
\[
\sigma_1\sigma_2\cdots\sigma_{w_i-1}
\]
also involves only the first \(m-1\) strands. Again, the induced permutation
fixes the \(m\)-th strand.

Therefore the induced permutation has at least two cycles, so the closure has
at least two components. This contradicts the assumption that the closure is a
knot. Hence \(s_m>m\).
\end{proof}

We next estimate the minimum possible number of crossings of a \(W\)-braid.

\begin{lemma}\label{lem:minimum-crossings-W}
Let \(W\) be a non-degenerate \(W\)-braid on \(p\geq 2\) strands. Then
\[
c(W)\geq p(p-1)+\frac{(p-1)(p-2)}{2}
=
\frac{(p-1)(3p-2)}{2}.
\]
Moreover, if the closure of \(W\) is a knot, then
\[
c(W)\geq (p+1)(p-1)+\frac{(p-1)(p-2)}{2}
=
\frac{3p(p-1)}{2}.
\]
\end{lemma}

\begin{proof}
By definition, a non-degenerate \(W\)-braid on \(p\geq 2\) strands has the
form
\[
W=
(\sigma_1\sigma_2\cdots\sigma_{w_1-1})
\cdots
(\sigma_1\sigma_2\cdots\sigma_{w_k-1})
A_1^{a_2}A_2^{a_3}\cdots A_{p-1}^{a_p},
\]
where the initial product may be empty,
\[
a_2,\dots,a_{p-1}>0,
\qquad
a_p\geq p.
\]
If the initial product is nonempty, then also
\[
a_p>w_1.
\]

The horseshoe tail alone contributes at least
\[
1+2+\cdots+(p-2)+p(p-1)
\]
crossings. Therefore
\[
c(W)\geq
\frac{(p-1)(p-2)}{2}+p(p-1)
=
\frac{(p-1)(3p-2)}{2}.
\]

Now assume that the closure of \(W\) is a knot. By
Lemma~\ref{lem:W-knot-last-exponent}, we have
\[
a_p>p.
\]
Since \(a_p\) is an integer, \(a_p\geq p+1\). Hence the last horseshoe factor
\[
A_{p-1}^{a_p}=(\sigma_{p-1}\cdots\sigma_1)^{a_p}
\]
contributes at least
\[
(p+1)(p-1)
\]
crossings. Therefore
\[
c(W)\geq
\frac{(p-1)(p-2)}{2}+(p+1)(p-1)
=
\frac{3p(p-1)}{2}. \qedhere
\]
\end{proof}

We can now prove the crossing-number obstruction.

\begin{theorem}\label{thm:V-not-W-crossing-obstruction}
Let \(p,q\) be integers with \(2\leq p\leq q\), and let
\[
V=
V((u_1,\overline{v_1}),\dots,(u_m,\overline{v_m}),
(r_1,s_1),\dots,(r_n,s_n),(p,q))
\]
be a \(V\)-link, where the sequences \(u_1,\dots,u_m\) and
\(r_1,\dots,r_n\) are allowed to be empty. Whenever they are nonempty, assume
that
\[
2\leq u_1<\cdots<u_m\leq p
\]
and
\[
2\leq r_1<\cdots<r_n<p.
\]
Set
\[
E(V)=
(q-p)(p-1)
+\sum_{i=1}^{m}v_i(u_i-1)
+\sum_{j=1}^{n}s_j(r_j-1).
\]
Then the following hold.
\begin{enumerate}
\item If
\[
E(V)<\frac{(p-1)(p-2)}{2},
\]
then \(V\) is not a \(W\)-link.

\item If \(V\) is a knot and
\[
E(V)<\frac{p(p-1)}{2},
\]
then \(V\) is not a \(W\)-link.
\end{enumerate}
\end{theorem}

\begin{proof}
Let \(c(V)\) denote the number of crossings of the defining \(V\)-braid. Then
\[
c(V)
=
q(p-1)
+\sum_{i=1}^{m}v_i(u_i-1)
+\sum_{j=1}^{n}s_j(r_j-1).
\]
Equivalently,
\[
c(V)=p(p-1)+E(V).
\]

We first prove item (1). Suppose that
\[
E(V)<\frac{(p-1)(p-2)}{2}.
\]
Then
\[
c(V)<
p(p-1)+\frac{(p-1)(p-2)}{2}
=
\frac{(p-1)(3p-2)}{2}.
\]

Assume, for contradiction, that \(V\) is a \(W\)-link. By
\cite[Theorem~3.9]{de2024lorenz}, the defining \(V\)-braid is minimal. Also,
every \(W\)-braid is minimal by Lemma~\ref{lem:W-braids-are-minimal}. Hence
any \(W\)-braid representing \(V\) must have the same number of strands as the
defining \(V\)-braid, namely \(p\).

By Lemma~\ref{lem:minimum-crossings-W}, every non-degenerate \(W\)-braid on
\(p\) strands has at least
\[
\frac{(p-1)(3p-2)}{2}
\]
crossings. Therefore any \(W\)-braid representative of \(V\) has strictly more
crossings than the defining \(V\)-braid. This contradicts
Lemma~\ref{lem:crossing-number-positive-same-strands}, since the defining
\(V\)-braid and the \(W\)-braid would be positive braid representatives of the
same link with the same number of strands. Thus \(V\) is not a \(W\)-link.

We now prove item (2). Suppose that \(V\) is a knot and
\[
E(V)<\frac{p(p-1)}{2}.
\]
Then
\[
c(V)<
p(p-1)+\frac{p(p-1)}{2}
=
\frac{3p(p-1)}{2}.
\]

Assume again, for contradiction, that \(V\) is a \(W\)-link. Since \(V\) is a
knot, any \(W\)-braid representing \(V\) has closure a knot. As before, the
minimality of the defining \(V\)-braid and of \(W\)-braids implies that such a
\(W\)-braid has \(p\) strands.

By the knot case of Lemma~\ref{lem:minimum-crossings-W}, every non-degenerate
\(W\)-braid on \(p\) strands whose closure is a knot has at least
\[
\frac{3p(p-1)}{2}
\]
crossings. Therefore any \(W\)-braid representative of \(V\) has strictly more
crossings than the defining \(V\)-braid. This again contradicts
Lemma~\ref{lem:crossing-number-positive-same-strands}.

Hence \(V\) is not a \(W\)-link.
\end{proof}

\subsection{Immediate consequences}

The obstruction immediately gives many \(V\)-links which are not \(W\)-links.

\begin{corollary}\label{cor:many-V-not-W}
Let \(p\geq 3\). Consider a \(V\)-link
\[
V((u_1,\overline{v_1}),\dots,(u_m,\overline{v_m}),
(r_1,s_1),\dots,(r_n,s_n),(p,p)),
\]
where the sequences \(u_1,\dots,u_m\) and \(r_1,\dots,r_n\) are allowed to be
empty. Whenever they are nonempty, assume that
\[
2\leq u_1<\cdots<u_m\leq p
\]
and
\[
2\leq r_1<\cdots<r_n<p.
\]
If
\[
\sum_{i=1}^{m}v_i(u_i-1)
+\sum_{j=1}^{n}s_j(r_j-1)
<
\frac{(p-1)(p-2)}{2},
\]
then this \(V\)-link is not a \(W\)-link.
\end{corollary}

\begin{proof}
This is the special case \(q=p\) of
Theorem~\ref{thm:V-not-W-crossing-obstruction}.
\end{proof}

\begin{remark}
Corollary~\ref{cor:many-V-not-W} gives a large supply of \(V\)-links which are
not \(W\)-links. For fixed \(p\geq 3\), any choice of additional \(V\)-blocks
whose total crossing contribution is smaller than
\[
\frac{(p-1)(p-2)}{2}
\]
gives a \(V\)-link which cannot be represented by a \(W\)-braid.
\end{remark}

The same obstruction also recovers a Holmes--Williams type obstruction for
torus links.

\begin{proposition}\label{prop:torus-link-crossing-obstruction}
Let \(2\leq p<q\). If the torus link \(T(p,q)\) is a \(W\)-link, then
\[
2q\geq 3p-2.
\]
If \(T(p,q)\) is a knot and is a \(W\)-link, then
\[
2q\geq 3p.
\]
\end{proposition}

\begin{proof}
The torus link \(T(p,q)\) is represented by the positive \(p\)-braid
\[
(\sigma_1\sigma_2\cdots\sigma_{p-1})^q.
\]
In the notation of Theorem~\ref{thm:V-not-W-crossing-obstruction}, this is the
\(V\)-link \(V((p,q))\), and
\[
E(V)=(q-p)(p-1).
\]

If
\[
2q<3p-2,
\]
then
\[
(q-p)(p-1)<\frac{(p-1)(p-2)}{2}.
\]
By Theorem~\ref{thm:V-not-W-crossing-obstruction}, \(T(p,q)\) is not a
\(W\)-link. Hence, if \(T(p,q)\) is a \(W\)-link, then
\[
2q\geq 3p-2.
\]

If \(T(p,q)\) is a knot and
\[
2q<3p,
\]
then
\[
(q-p)(p-1)<\frac{p(p-1)}{2}.
\]
The knot case of Theorem~\ref{thm:V-not-W-crossing-obstruction} implies that
\(T(p,q)\) is not a \(W\)-link. Therefore, if the torus knot \(T(p,q)\) is a
\(W\)-link, then
\[
2q\geq 3p. \qedhere
\]
\end{proof}

\begin{remark}
Using the characterization of horseshoe links by \(W\)-braids,
Proposition~\ref{prop:torus-link-crossing-obstruction} recovers the
Holmes--Williams numerical obstruction for torus knots: if \(T(p,q)\), with
\(p<q\), is a horseshoe knot, then
\[
3p\leq 2q.
\]
For torus links with more than one component, the same crossing-number
mechanism gives the link analogue
\[
2q\geq 3p-2.
\]
\end{remark}

\subsection{Hyperbolic \(V\)-knots which are not \(W\)-links}

We now apply the obstruction to construct hyperbolic examples.

\begin{lemma}\label{lem:hyperbolic-V-2-s1-p-pplus1}
Let \(p\geq 4\), and let \(s_1\) be a positive even integer. Then the
\(V\)-link
\[
V((2,s_1),(p,p+1))
\]
is a hyperbolic knot.
\end{lemma}

\begin{proof}
First, the assumption that \(s_1\) is even implies that this \(V\)-link is a
knot. Indeed, the factor \((\sigma_1)^{s_1}\) induces the identity
permutation on the strands, while
\[
(\sigma_1\sigma_2\cdots\sigma_{p-1})^{p+1}
\]
induces a \(p\)-cycle, since \(\gcd(p,p+1)=1\). Hence the induced permutation
of the whole braid is a \(p\)-cycle, and the closure is a knot.

If \(s_1\geq 4\), then the result follows from
\cite[Theorem~1.1]{MR3801447}. It remains to consider the case \(s_1=2\).
In this case, the result follows from \cite[Corollary~8.6]{de2024lorenz}.
Hence \(V((2,s_1),(p,p+1))\) is hyperbolic for every positive even integer
\(s_1\).
\end{proof}

\begin{theorem}\label{thm:hyperbolic-V-2-s1-p-pplus1-not-W}
Let \(s_1\) be a positive even integer, and let \(p\geq 4\). Suppose that
\[
(p-1)(p-2)>2s_1.
\]
Then the \(V\)-knot
\[
V((2,s_1),(p,p+1))
\]
is hyperbolic and is not a \(W\)-link. Consequently, it cannot be embedded in
the horseshoe template, and hence it is not a R\"ossler link.
\end{theorem}

\begin{proof}
By Lemma~\ref{lem:hyperbolic-V-2-s1-p-pplus1}, the \(V\)-link
\[
V((2,s_1),(p,p+1))
\]
is a hyperbolic knot.

It remains to show that it is not a \(W\)-link. For this \(V\)-knot, we have
\[
q=p+1,\qquad r_1=2.
\]
Thus
\[
E(V)
=
(q-p)(p-1)+s_1(r_1-1)
=
p-1+s_1.
\]
The hypothesis
\[
(p-1)(p-2)>2s_1
\]
is equivalent to
\[
p-1+s_1<\frac{p(p-1)}{2}.
\]
Since \(V((2,s_1),(p,p+1))\) is a knot, the knot case of
Theorem~\ref{thm:V-not-W-crossing-obstruction} implies that it is not a
\(W\)-link.

Finally, by Theorem~\ref{thm:minimal-horseshoe-W}, every horseshoe link admits
a minimal \(W\)-braid representative. Therefore, if this link were a horseshoe
link, it would be a \(W\)-link, contradicting what we have just proved. Hence
it is not a horseshoe link, and therefore it is not a R\"ossler link.
\end{proof}

\begin{corollary}\label{cor:infinitely-many-hyperbolic-V-not-W}
There are infinitely many hyperbolic \(V\)-knots which are not \(W\)-links.
More precisely, for every \(p\geq 4\), the \(V\)-knot
\[
V((2,2),(p,p+1))
\]
is hyperbolic and is not a \(W\)-link. Consequently, it is not a horseshoe link
and hence not a R\"ossler link.
\end{corollary}

\begin{proof}
By Lemma~\ref{lem:hyperbolic-V-2-s1-p-pplus1}, the \(V\)-link
\[
V((2,2),(p,p+1))
\]
is a hyperbolic knot for every \(p\geq 4\).

For
\[
V=V((2,2),(p,p+1)),
\]
we have
\[
q=p+1,\qquad r_1=2,\qquad s_1=2.
\]
Thus
\[
E(V)=(q-p)(p-1)+s_1(r_1-1)=p+1.
\]
Since
\[
p+1<\frac{p(p-1)}{2}
\]
for every \(p\geq 4\), the knot case of
Theorem~\ref{thm:V-not-W-crossing-obstruction} implies that
\[
V((2,2),(p,p+1))
\]
is not a \(W\)-link. Therefore it is not a horseshoe link by
Theorem~\ref{thm:minimal-horseshoe-W}, and hence it is not a R\"ossler link.

The knots in this family are pairwise distinct because their defining
\(V\)-braids have \(p\) strands and are minimal braid representatives. Thus
their braid indices are equal to \(p\), and different values of \(p\) give
different knots.
\end{proof}

\subsection{Hyperbolic Lorenz links with more than one component}

We now construct hyperbolic examples with more than one component.

\begin{lemma}\label{lem:hyperbolic-V-3-3s1-2p-2pplus2}
Let \(p\geq 3\) be a prime number, and let \(s_1>3\) be an integer. Then the
\(V\)-link
\[
V((3,3s_1),(2p,2p+2))
\]
is hyperbolic.
\end{lemma}

\begin{proof}
This follows directly from \cite[Theorem~1.1]{Pai1}.
\end{proof}

\begin{theorem}\label{thm:infinitely-many-hyperbolic-Lorenz-links-not-horseshoe}
There are infinitely many hyperbolic Lorenz links with more than one component
which are not horseshoe links. More precisely, for every prime \(p\geq 7\), the
\(V\)-link
\[
V((3,12),(2p,2p+2))
\]
is a hyperbolic Lorenz link with two components, and it cannot be embedded in
the horseshoe template. Consequently, it is not a R\"ossler link.
\end{theorem}

\begin{proof}
By Lemma~\ref{lem:hyperbolic-V-3-3s1-2p-2pplus2}, applied with \(s_1=4\), the
\(V\)-link
\[
V((3,12),(2p,2p+2))
\]
is hyperbolic for every prime \(p\geq 7\). Since \(V\)-links are Lorenz links,
this gives a hyperbolic Lorenz link.

We now show that this link has two components. Its defining braid is
\[
(\sigma_1\sigma_2)^{12}
(\sigma_1\sigma_2\cdots\sigma_{2p-1})^{2p+2}.
\]
The factor \((\sigma_1\sigma_2)^{12}\) induces the identity permutation,
because \(\sigma_1\sigma_2\) induces a \(3\)-cycle and \(12\) is divisible by
\(3\). Hence the number of components is determined by the permutation induced
by
\[
(\sigma_1\sigma_2\cdots\sigma_{2p-1})^{2p+2}.
\]
The braid \(\sigma_1\sigma_2\cdots\sigma_{2p-1}\) induces a \(2p\)-cycle. Since
\[
\gcd(2p,2p+2)=2,
\]
the induced permutation has two cycles. Therefore
\[
V((3,12),(2p,2p+2))
\]
has two components.

It remains to prove that this link is not a horseshoe link. Let
\[
V=V((3,12),(2p,2p+2)).
\]
In the notation of Theorem~\ref{thm:V-not-W-crossing-obstruction}, the number
of strands is \(2p\), and
\[
q=2p+2,\qquad r_1=3,\qquad s_1=12.
\]
Thus
\[
E(V)
=
(q-2p)(2p-1)+s_1(r_1-1)
=
2(2p-1)+12(3-1)
=
4p+22.
\]
On the other hand,
\[
\frac{(2p-1)(2p-2)}{2}
=
(2p-1)(p-1).
\]
For every prime \(p\geq 7\), we have
\[
4p+22<(2p-1)(p-1).
\]
Therefore
\[
E(V)<\frac{(2p-1)(2p-2)}{2}.
\]
By Theorem~\ref{thm:V-not-W-crossing-obstruction}, \(V\) is not a \(W\)-link.

Finally, by Theorem~\ref{thm:minimal-horseshoe-W}, every horseshoe link admits
a minimal \(W\)-braid representative. Therefore, if \(V\) could be embedded in
the horseshoe template, then it would be a \(W\)-link, contradicting what we
have just proved. Thus
\[
V((3,12),(2p,2p+2))
\]
is not a horseshoe link, and hence not a R\"ossler link.

Since there are infinitely many primes \(p\geq 7\), this gives infinitely many
examples. Moreover, these links are pairwise distinct: their defining
\(V\)-braids are minimal braid representatives and have \(2p\) strands. Hence
their braid indices are equal to \(2p\), and different values of \(p\) give
different links.
\end{proof}

\subsection{Satellite Lorenz links which are not horseshoe links}

We next obtain satellite examples.

\begin{lemma}\label{lem:V-2-s1-p-pplus2-satellite}
Let \(p\geq 4\) be even, and let \(s_1\) be a positive integer. Then the
\(V\)-link
\[
V((2,s_1),(p,p+2))
\]
is a satellite link.
\end{lemma}

\begin{proof}
This follows from \cite[Theorem~6.1]{de2024lorenz}.
\end{proof}

\begin{lemma}\label{lem:V-2-s1-p-pplus2-components}
Let \(p\geq 4\) be even, and let \(s_1\) be a positive integer. Then the
\(V\)-link
\[
V((2,s_1),(p,p+2))
\]
is a knot if \(s_1\) is odd, and has two components if \(s_1\) is even.
\end{lemma}

\begin{proof}
This component count also follows from the general results on twisted torus
links in \cite{2025twisted}. For completeness, we give a direct proof in this
special case.

The defining braid is
\[
(\sigma_1)^{s_1}(\sigma_1\sigma_2\cdots\sigma_{p-1})^{p+2}.
\]
Let \(c\) be the \(p\)-cycle induced by
\[
\sigma_1\sigma_2\cdots\sigma_{p-1}.
\]
Since \(p\) is even, we have
\[
c^{p+2}=c^2,
\]
and \(c^2\) has two cycles, one on the odd-numbered strands and one on the
even-numbered strands.

If \(s_1\) is even, then \((\sigma_1)^{s_1}\) induces the identity
permutation. Hence the induced permutation has two cycles, and the closure has
two components.

If \(s_1\) is odd, then \((\sigma_1)^{s_1}\) induces the transposition
\((1\,2)\). The strands \(1\) and \(2\) lie in the two different cycles of
\(c^2\). Multiplying by this transposition merges the two cycles into one.
Hence the induced permutation has one cycle, and the closure is a knot.
\end{proof}

\begin{theorem}\label{thm:satellite-Lorenz-not-Rossler}
Let \(p\geq 4\) be even, and let \(s_1\) be a positive integer. Then
\[
V((2,s_1),(p,p+2))
\]
is a satellite Lorenz link. Moreover, the following hold.

\begin{enumerate}
\item If \(s_1\) is odd and
\[
s_1<\frac{(p-1)(p-4)}{2},
\]
then \(V((2,s_1),(p,p+2))\) is a satellite Lorenz knot which is not a
\(W\)-link. In particular, it cannot be embedded in the horseshoe template,
and hence it is not a R\"ossler link.

\item If \(s_1\) is even and
\[
s_1<\frac{(p-1)(p-6)}{2},
\]
then \(V((2,s_1),(p,p+2))\) is a two-component satellite Lorenz link which is
not a \(W\)-link. In particular, it cannot be embedded in the horseshoe
template, and hence it is not a R\"ossler link.
\end{enumerate}
\end{theorem}

\begin{proof}
By Lemma~\ref{lem:V-2-s1-p-pplus2-satellite}, the link
\[
V((2,s_1),(p,p+2))
\]
is a satellite link. Since \(V\)-links are Lorenz links, it is a satellite
Lorenz link.

By Lemma~\ref{lem:V-2-s1-p-pplus2-components}, this \(V\)-link is a knot when
\(s_1\) is odd and has two components when \(s_1\) is even.

It remains to prove that it is not a \(W\)-link under the stated hypotheses.
Let
\[
V=V((2,s_1),(p,p+2)).
\]
Then
\[
q=p+2,\qquad r_1=2.
\]
Thus, in the notation of Theorem~\ref{thm:V-not-W-crossing-obstruction},
\[
E(V)
=
(q-p)(p-1)+s_1(r_1-1)
=
2(p-1)+s_1.
\]

Suppose first that \(s_1\) is odd and
\[
s_1<\frac{(p-1)(p-4)}{2}.
\]
Then
\[
E(V)
=
2(p-1)+s_1
<
2(p-1)+\frac{(p-1)(p-4)}{2}
=
\frac{p(p-1)}{2}.
\]
Since \(V\) is a knot, the knot case of
Theorem~\ref{thm:V-not-W-crossing-obstruction} implies that \(V\) is not a
\(W\)-link.

Now suppose that \(s_1\) is even and
\[
s_1<\frac{(p-1)(p-6)}{2}.
\]
Then
\[
E(V)
=
2(p-1)+s_1
<
2(p-1)+\frac{(p-1)(p-6)}{2}
=
\frac{(p-1)(p-2)}{2}.
\]
Therefore, by Theorem~\ref{thm:V-not-W-crossing-obstruction}, \(V\) is not a
\(W\)-link.

Finally, by Theorem~\ref{thm:minimal-horseshoe-W}, if \(V\) could be embedded
in the horseshoe template, then it would admit a \(W\)-braid representative.
This contradicts what we have just proved. Hence \(V\) is not a horseshoe
link, and therefore it is not a R\"ossler link.
\end{proof}

\begin{corollary}\label{cor:infinitely-many-satellite-Lorenz-not-Rossler}
There are infinitely many satellite Lorenz knots which are not R\"ossler links.
More precisely, for every even integer \(p\geq 6\), the knot
\[
V((2,1),(p,p+2))
\]
is a satellite Lorenz knot which cannot be embedded in the horseshoe template.
\end{corollary}

\begin{proof}
Take \(s_1=1\) in Theorem~\ref{thm:satellite-Lorenz-not-Rossler}. Since
\(s_1\) is odd, the link
\[
V((2,1),(p,p+2))
\]
is a knot. Moreover, for every even integer \(p\geq 6\), we have
\[
1<\frac{(p-1)(p-4)}{2}.
\]
Therefore Theorem~\ref{thm:satellite-Lorenz-not-Rossler} implies that
\[
V((2,1),(p,p+2))
\]
is a satellite Lorenz knot which is not a horseshoe link, and hence not a
R\"ossler link.

These knots are pairwise distinct as \(p\) varies. Indeed, their defining
\(V\)-braids have \(p\) strands and are minimal braid representatives. Hence
their braid indices are equal to \(p\). Therefore different values of \(p\)
give knots with different braid indices.
\end{proof}

\subsection{A sublink obstruction}

Finally, we observe that the obstruction is inherited by superlinks.

\begin{proposition}\label{prop:sublink-obstruction-horseshoe}
Let \(L\subset S^3\) be a link. Suppose that \(L\) contains a sublink
\(L'\subset L\) which is not a horseshoe link. Then \(L\) is not a horseshoe
link. Consequently, since R\"ossler links and horseshoe links are carried by
equivalent templates, if \(L'\) is not a R\"ossler link, then any link
containing \(L'\) as a sublink is not a R\"ossler link.
\end{proposition}

\begin{proof}
We prove the contrapositive. Suppose that \(L\) is a horseshoe link. Then
\(L\) can be embedded in the horseshoe template. Since \(L'\) is a sublink of
\(L\), it is obtained by deleting some components of \(L\). Deleting components
from a link embedded in a template leaves a link still embedded in the same
template. Hence \(L'\) can also be embedded in the horseshoe template.
Therefore \(L'\) is a horseshoe link.

Thus, if \(L'\) is not a horseshoe link, then \(L\) cannot be a horseshoe link.
Since R\"ossler links and horseshoe links are carried by equivalent templates,
the same argument gives the corresponding statement for R\"ossler links.
\end{proof}

\begin{corollary}\label{cor:more-Lorenz-not-horseshoe}
Let \(L\) be a Lorenz link. Suppose that \(L\) contains a sublink isotopic to
one of the Lorenz links constructed above which is not a horseshoe link. Then
\(L\) is not a horseshoe link, and hence it is not a R\"ossler link.
\end{corollary}

\begin{proof}
The given sublink is not a horseshoe link. Therefore the result follows
immediately from Proposition~\ref{prop:sublink-obstruction-horseshoe}.
\end{proof}

\bibliographystyle{amsplain}  

\bibliography{References}

@article {unexpected,
    AUTHOR = {de Paiva, Thiago},
     TITLE = {Unexpected essential surfaces among exteriors of twisted torus
              knots},
   JOURNAL = {Algebr. Geom. Topol.},
  FJOURNAL = {Algebraic \& Geometric Topology},
    VOLUME = {22},
      YEAR = {2022},
    NUMBER = {8},
     PAGES = {3965--3982},
      ISSN = {1472-2747,1472-2739},
   MRCLASS = {57K10 (57K35)},
  MRNUMBER = {4562562},
       DOI = {10.2140/agt.2022.22.3965},
       URL = {https://doi.org/10.2140/agt.2022.22.3965},
}

@article {twofulltwists,
    AUTHOR = {de Paiva, Thiago},
     TITLE = {Hyperbolic knots given by positive braids with at least two
              full twists},
   JOURNAL = {Proc. Amer. Math. Soc.},
  FJOURNAL = {Proceedings of the American Mathematical Society},
    VOLUME = {150},
      YEAR = {2022},
    NUMBER = {12},
     PAGES = {5449--5458},
      ISSN = {0002-9939,1088-6826},
   MRCLASS = {57K10 (20F36 57K32)},
  MRNUMBER = {4494619},
MRREVIEWER = {L.\ Neuwirth},
       DOI = {10.1090/proc/16035},
       URL = {https://doi.org/10.1090/proc/16035},
}

@article {de2021satellites,
    AUTHOR = {de Paiva, Thiago and Purcell, Jessica S.},
     TITLE = {Satellites and {L}orenz knots},
   JOURNAL = {Int. Math. Res. Not. IMRN},
  FJOURNAL = {International Mathematics Research Notices. IMRN},
      YEAR = {2023},
    NUMBER = {19},
     PAGES = {16540--16573},
      ISSN = {1073-7928,1687-0247},
   MRCLASS = {57K10},
  MRNUMBER = {4651895},
       DOI = {10.1093/imrn/rnac335},
       URL = {https://doi.org/10.1093/imrn/rnac335},
}

@article {dePaivaPurcell2024,
    AUTHOR = {de Paiva, Thiago and Purcell, Jessica S.},
     TITLE = {Hyperbolic and satellite {L}orenz links obtained by twisting},
   JOURNAL = {Michigan Math. J.},
  FJOURNAL = {Michigan Mathematical Journal},
      YEAR = {2024},
    VOLUME = {73},
    NUMBER = {2},
     PAGES = {379--404},
      ISSN = {0026-2285},
   MRCLASS = {57K10 (57K32 57M25)},
       DOI = {10.1307/mmj/20246510},
       URL = {https://doi.org/10.1307/mmj/20246510}
}

@article {newtwis,
    AUTHOR = {Birman, Joan and Kofman, Ilya},
     TITLE = {A new twist on {L}orenz links},
   JOURNAL = {J. Topol.},
  FJOURNAL = {Journal of Topology},
    VOLUME = {2},
      YEAR = {2009},
    NUMBER = {2},
     PAGES = {227--248},
      ISSN = {1753-8416},
   MRCLASS = {57M25 (57M27 57M50)},
  MRNUMBER = {2529294},
MRREVIEWER = {Daniel Matei},
       DOI = {10.1112/jtopol/jtp007},
       URL = {https://doi-org.ezproxy.lib.monash.edu.au/10.1112/jtopol/jtp007},
}

@article {de2022torus,
    AUTHOR = {de Paiva, Thiago},
     TITLE = {Torus {L}orenz links obtained by full twists along torus
              links},
   JOURNAL = {Proc. Amer. Math. Soc.},
  FJOURNAL = {Proceedings of the American Mathematical Society},
    VOLUME = {151},
      YEAR = {2023},
    NUMBER = {6},
     PAGES = {2671--2677},
      ISSN = {0002-9939,1088-6826},
   MRCLASS = {57K10 (57K32 57K35)},
  MRNUMBER = {4576328},
       DOI = {10.1090/proc/16261},
       URL = {https://doi.org/10.1090/proc/16261},
}

@article{birman1983knotted,
  title={Knotted periodic orbits in dynamical system. II. Knot holders for fibered knots},
  author={Birman, Joan S and Williams, Robert F},
  journal={Contemporary Mathematics},
  volume={20},
  pages={1--60},
  year={1983},
  publisher={American Mathematical Society}
}

@article {periodicorbits,
    AUTHOR = {Birman, Joan S. and Williams, R. F.},
     TITLE = {Knotted periodic orbits in dynamical systems. {I}. {L}orenz's
              equations},
   JOURNAL = {Topology},
  FJOURNAL = {Topology. An International Journal of Mathematics},
    VOLUME = {22},
      YEAR = {1983},
    NUMBER = {1},
     PAGES = {47--82},
      ISSN = {0040-9383},
   MRCLASS = {58F13 (57M25)},
  MRNUMBER = {682059},
MRREVIEWER = {Hans G. Bothe},
       DOI = {10.1016/0040-9383(83)90045-9},
       URL = {https://doi.org/10.1016/0040-9383(83)90045-9},
}

@article {dePaiva2025,
    AUTHOR = {de Paiva, Thiago},
     TITLE = {Satellite knots that cannot be represented by positive braids with full twists},
   JOURNAL = {New York J. Math.},
  FJOURNAL = {New York Journal of Mathematics},
      YEAR = {2025},
    VOLUME = {31},
     PAGES = {1690--1701},
      ISSN = {1076-9803},
   MRCLASS = {57K10 (57K35 57M25)},
       URL = {https://nyjm.albany.edu/j/2025/31-67.html}
}

@article {MR3801447,
    AUTHOR = {Lee, Sangyop},
     TITLE = {Satellite knots obtained by twisting torus knots:
              hyperbolicity of twisted torus knots},
   JOURNAL = {Int. Math. Res. Not. IMRN},
  FJOURNAL = {International Mathematics Research Notices. IMRN},
      YEAR = {2018},
    NUMBER = {3},
     PAGES = {785--815},
      ISSN = {1073-7928,1687-0247},
   MRCLASS = {57M25},
  MRNUMBER = {3801447},
MRREVIEWER = {Toshio\ Saito},
       DOI = {10.1093/imrn/rnw255},
       URL = {https://doi.org/10.1093/imrn/rnw255},
}

@article{de2024lorenz,
  title={Lorenz links, {T}-links, Minimal Braids, Positive braids with a full twist, and Geometric Types},
  author={de Paiva, Thiago},
  journal={arXiv preprint arXiv:2409.14824},
  year={2024}
}

@article{Volumebounds,
  title={Volumes, {L}orenz-like templates, and braids},
  author={de Paiva, Thiago and Hui, Connie On Yu and Andr\'es Rodr\'iguez Migueles, Jos\'e},
  journal={arXiv preprint arXiv:2410.04391},
  year={2024}
}

@article{2025twisted,
  title={On the number of components of twisted torus links},
  author={Adnan and de Paiva, Thiago and Park, Kyungbae},
  journal={arXiv preprint arXiv:2505.01021},
  year={2025}
}

@misc{KofmanHorseshoeBeamer,
  author       = {Ilya Kofman},
  title        = {Lorenz and horseshoe knots},
  note         = {Lecture slides, Brigham Young University},
  year         = {2013},
  url          = {https://www.math.csi.cuny.edu/~ikofman/horseshoe_beamer_BYU.pdf}
}

@article {Pai1,
    AUTHOR = {de Paiva, Thiago},
     TITLE = {Hyperbolic twisted torus links},
   JOURNAL = {Geom. Dedicata},
  FJOURNAL = {Geometriae Dedicata},
    VOLUME = {217},
      YEAR = {2023},
    NUMBER = {2},
     PAGES = {Paper No. 42, 16},
      ISSN = {0046-5755,1572-9168},
   MRCLASS = {57K10 (57K32)},
  MRNUMBER = {4551666},
MRREVIEWER = {Brandy\ Guntel\ Doleshal},
       DOI = {10.1007/s10711-023-00777-z},
       URL = {https://doi.org/10.1007/s10711-023-00777-z},
}

@article {Franks,
    AUTHOR = {Franks, John and Williams, R. F.},
     TITLE = {Braids and the {J}ones polynomial},
   JOURNAL = {Trans. Amer. Math. Soc.},
  FJOURNAL = {Transactions of the American Mathematical Society},
    VOLUME = {303},
      YEAR = {1987},
    NUMBER = {1},
     PAGES = {97--108},
      ISSN = {0002-9947},
   MRCLASS = {57M25},
  MRNUMBER = {896009},
MRREVIEWER = {Kenneth C. Millett},
       DOI = {10.2307/2000780},
       URL = {https://doi-org.ezproxy.lib.monash.edu.au/10.2307/2000780},
}

@inproceedings {Genus,
    AUTHOR = {Stallings, John R.},
     TITLE = {Constructions of fibred knots and links},
 BOOKTITLE = {Algebraic and geometric topology ({P}roc. {S}ympos. {P}ure
              {M}ath., {S}tanford {U}niv., {S}tanford, {C}alif., 1976),
              {P}art 2},
    SERIES = {Proc. Sympos. Pure Math., XXXII},
     PAGES = {55--60},
 PUBLISHER = {Amer. Math. Soc., Providence, R.I.},
      YEAR = {1978},
   MRCLASS = {57M25},
  MRNUMBER = {520522},
MRREVIEWER = {Kenneth A. Perko, Jr.},
}

@article{HolmesWilliams1985,
  author  = {Holmes, Philip and Williams, Robert F.},
  title   = {Knotted periodic orbits in suspensions of Smale's horseshoe:
             torus knots and bifurcation sequences},
  journal = {Archive for Rational Mechanics and Analysis},
  volume  = {90},
  number  = {2},
  pages   = {115--194},
  year    = {1985}
}

@article{deCarvalhoHall2002Forcing,
  author  = {de Carvalho, Andr{\'e} and Hall, Toby},
  title   = {The forcing relation for horseshoe braid types},
  journal = {Experimental Mathematics},
  volume  = {11},
  number  = {2},
  pages   = {271--288},
  year    = {2002}
}

@article{deCarvalhoHall2010Decoration,
  author  = {de Carvalho, Andr{\'e} and Hall, Toby},
  title   = {Decoration invariants for horseshoe braids},
  journal = {Discrete and Continuous Dynamical Systems},
  volume  = {27},
  number  = {3},
  pages   = {863--906},
  year    = {2010}
}

@article{deCarvalhoHall2003Conjugacies,
  author  = {de Carvalho, Andr{\'e} and Hall, Toby},
  title   = {Conjugacies between horseshoe braids},
  journal = {Nonlinearity},
  volume  = {16},
  number  = {4},
  pages   = {1329--1338},
  year    = {2003}
}

@article{Igra2025Rossler,
  author  = {Igra, Eran},
  title   = {Knots and chaos in the R{\"o}ssler system},
  journal = {Journal of Differential Equations},
  volume  = {437},
  pages   = {113290},
  year    = {2025},
  doi     = {10.1016/j.jde.2025.113290}
}

@article{LetellierGouesbet1996,
  author  = {Letellier, Christophe and Gouesbet, G{\'e}rard},
  title   = {Topological characterization of reconstructed attractors modding
             out symmetries},
  journal = {Journal de Physique II},
  volume  = {6},
  pages   = {1615--1638},
  year    = {1996}
}

@article{RosalieLetellier2013,
  author  = {Rosalie, Martin and Letellier, Christophe},
  title   = {Systematic template extraction from chaotic attractors: {I}.
             Genus-one attractors with an inversion symmetry},
  journal = {Journal of Physics A: Mathematical and Theoretical},
  volume  = {46},
  number  = {37},
  pages   = {375101},
  year    = {2013},
  doi     = {10.1088/1751-8113/46/37/375101}
}

@article{LetellierDutertreMaheu1995,
  author  = {Letellier, Christophe and Dutertre, Pierre and Maheu, Bertrand},
  title   = {Unstable periodic orbits and templates of the {R}\"ossler system:
             toward a systematic topological characterization},
  journal = {Chaos: An Interdisciplinary Journal of Nonlinear Science},
  volume  = {5},
  number  = {1},
  pages   = {271--282},
  year    = {1995},
  doi     = {10.1063/1.166076}
}

\end{document}